\documentclass[a4paper,12pt]{amsart}
\usepackage[shortlabels]{enumitem}
\usepackage[latin1]{inputenc}
\usepackage{amsfonts}
\usepackage{amsmath, latexsym}
\usepackage{amssymb,verbatim}
\usepackage{color}
\usepackage{euscript}

\newtheorem{thm}{Theorem}
\newtheorem{lem}[thm]{Lemma}

\newtheorem{claim}[thm]{Claim}

\newenvironment{proof*}{\vskip 2mm\noindent {}}{\hfill $\Box$ \vskip 2mm}

\newcommand{\ba}{\begin{align*}}
\newcommand{\ea}{\end{align*}}

\newcommand{\R}{\mathbb{R}}
\newcommand{\Ss}{\mathbb{S}}

\newcommand{\h}{\mathcal{H}}

\newcommand{\f}{\varphi}

\newcommand{\be}{\begin{equation}}
\newcommand{\ee}{\end{equation}}

\title{A short proof of the $\mathcal C^{1,1}$ regularity for the eikonal equation}

\author{Radu Ignat}
\address{Institut de Math\'ematiques de Toulouse, UMR 5219,
Universit\'e de Toulouse, CNRS, UPS
IMT,
F-31062 Toulouse Cedex 9, France.}
\email{Radu.Ignat@math.univ-toulouse.fr}

\date{\today}

\begin{document}

\maketitle

\vspace{-0.8cm}

\begin{abstract}
\smallskip
We give a short and self-contained proof of the interior $\mathcal C^{1,1}$ regularity of solutions $\f:\Omega \to \R$ to the eikonal equation $|\nabla \f|=1$ in an open set $\Omega\subset \mathbb{R}^{N}$ in dimension $N\geq 1$ under the assumption that $\f$ is pointwise differentiable in $\Omega$.
\end{abstract}

\section{Introduction}

The aim of this note is to give a short and self-contained proof of the following result known in the theory of Hamilton-Jacobi equations:
\begin{thm} 
\label{thm1}
Let $\Omega \subset \mathbb{R}^{N}$ be an open set in dimension $N\geq 1$ and  $\varphi: \Omega \rightarrow \mathbb{R}$ be a pointwise differentiable
solution to the eikonal equation $|\nabla \varphi|=1$ in $\Omega$. Then $\nabla \varphi$ is locally Lipschitz in $\Omega$.
\end{thm}

The usual (standard) proof of this result is based on the following steps (see e.g. Lions \cite{Lions}, Cannarsa-Sinestrari \cite{CS}): first, one checks that $\f$ (and $-\f$) is a viscosity solution to the eikonal equation (see \cite[Definition 5.2.1]{CS}); second, one proves that $\f$ is both semiconcave and semiconvex  with linear modulus (see \cite[Theorem 5.3.7]{CS}). Third, one proves that $\f$ is $C^1$ (see \cite[Theorem 3.3.7]{CS}) and finally, that $\f$ is locally $C^{1,1}$ in $\Omega$ (see \cite[Corollary 3.3.8]{CS}).

Our approach is based on the geometry of characteristics associated to the eikonal equation. More precisely, if $x_{0} \in \Omega$, we say that $X:=X_{x_0}$ is a characteristic of a solution $\f$ passing through $x_0$ in some time interval $t \in[-T, T]$ if
\be
\label{cara}
\left\{\begin{array}{l}\dot{X}(t)=\nabla \f(X(t)) \textrm{ for $t\in [-T,T]$},\\ X(0)=x_{0}.\end{array}\right.
\ee
Then the beautiful proof of Caffarelli-Crandall \cite[Lemma 2.2]{CC} shows in a short and self-contained manner that every point $x_0\in \Omega$ has a characteristic $X_{x_0}$ that is a straight line along which $\nabla \f$ is constant and $\f$ is affine. Finally, we give a geometric argument on the structure of characteristics yielding the locally Lipschitz regularity of $\nabla \f$ in $\Omega$.

The regularity result in Theorem \ref{thm1} is optimal: such solution $\f$ of the eikonal equation is not $C^2$ in general (see e.g. \cite[Proposition 1]{Ign}). We mention that a more general regularizing effect (i.e., $\nabla \f$ is locally Lipschitz away from vortex point singularities) is proved under a weaker assumption $\nabla \f\in W^{1/p,p}$ for $p\in [1,3]$, see \cite{Ign, DI}. Similar results are obtained in the context of the Aviles-Giga model which can be seen as a regularization of the eikonal equation (see \cite{JOP, GL}).

\section{Proof of the main result}

The first step is to show that each point $x_0\in \Omega$ has a characteristic $X:=X_{x_0}$ that is a straight line in direction $\nabla \f(x_0)$. Moreover, $\nabla \f$ is constant while $\f$ is affine along this characteristic. This fact yields $\f\in C^1(\Omega)$. In order to have a self-contained proof of Theorem \ref{thm1}, we repeat here the very nice argument of  
Caffarelli-Crandall \cite[Lemma 2.2]{CC} based on a maximum type principle for the eikonal equation.

\begin{lem}
\label{lem1}
Let $\Omega \subset \mathbb{R}^{N}$ be an open set and  $\varphi: \Omega \rightarrow \mathbb{R}$ be a 
pointwise differentiable solution of the eikonal equation $|\nabla \varphi|=1$ in $\Omega$. Then for every $x_0\in \Omega$, $X(t)=x_0+t\nabla \f(x_0)$ is a characteristic of \eqref{cara} and 
$$\nabla \varphi(X(t))=\nabla \varphi\left(x_{0}\right), \,  \f(X(t))=\f(x_0)+t, \, \forall t \in[-T, T]$$ for some $T>0$. As a consequence, $\f\in C^1(\Omega)$.
\end{lem}

\begin{proof}
This proof follows the lines in \cite[Lemma 2.2]{CC}. Let $R>0$ be such that $\bar B_R(x_0)\subset \Omega$ and consider
$$M_r=\max_{\bar B_r(x_0)} \f, \quad m_r=\min_{\bar B_r(x_0)} \f, \quad \forall r\in [0,R].$$

\begin{claim}
\label{ballext}
 $M_r=\f(x_0)+r$ and $m_r=\f(x_0)-r$ for every $r\in [0,R]$.
\end{claim}

\begin{proof}
 For $r\in [0,R]$, we pick some maximum point $x_r^+\in \bar B_r(x_0)$ such that $\f(x_r^+)=M_r$. First, we show that $r\in [0,R]\mapsto M_r$ is a nondecreasing $1$-Lipschitz function. Indeed, for $R\geq r>\tilde r$, as $|x_r^+-x_0|\leq r$, we can find a vector $e\in \R^N$ such that $|e|\leq r-\tilde r$ and $|x_r^++e-x_0|\leq \tilde r$, i.e., $x_r^++e\in \bar B_{\tilde r}(x_0)$; this yields
$$0\leq M_r-M_{\tilde r}\leq \f(x_r^+)-\f(x_r^++e)\leq |e|\leq r-\tilde r$$
because $\f$ is $1$-Lipschitz. Second, we prove that $\frac{d M_r}{dr}=1$ a.e. in $(0,R)$ because for $r\in (0,R)$ and
for small $h>0$, as $x_r^++ h\nabla \f(x_r^+)\in \bar B_{r+h}(x_0)$, we have
$$
\liminf_{h\to 0}\frac{M_{r+ h}-M_r}{h}\geq \liminf_{h\to 0} \frac{\f(x_r^+ + h\nabla \f(x_r^+))-\f(x_r^+)}{h}= |\nabla \f(x_r^+)|^2=1.
$$
As $M_0=\f(x_0)$, we conclude $M_r=\f(x_0)+r$. Up to changing $\f$ in $-\f$, one also gets $m_r=\f(x_0)-r$ for $r\in [0,R]$.
\end{proof}

\smallskip

 To conclude the proof of Lemma \ref{lem1}, pick some minimum point $x_R^-\in \bar B_R(x_0)$ such that $\f(x_R^-)=m_R=\f(x_0)-R$ (by  Claim \ref{ballext}). As $\f$ is $1$-Lipschitz,
 we have, again by  Claim \ref{ballext}:
 $$
 2R=\f(x_R^+)-\f(x_R^-)\leq |x_R^+-x_R^-|\leq 2R,
 $$
 which means that $[x_R^+, x_R^-]$ is a diameter in  $ \bar B_R(x_0)$. 
 Note that $x_R^+$ (resp. $x_R^-$) is the unique maximum (resp. minimum) of $\f$ in $\bar B_R(x_0)$
 because if $\tilde x_R^+$ is another maximum, then it has to be antipodal to  $x_R^-$, that is, $\tilde x_R^+=x_R^+$ (the same for the uniqueness of $x_R^-$). In particular, $e_*=\frac{x_R^+-x_0}R\in \Ss^{N-1}$. Define $g:[-R,R]\to \R$ by $g(r)=\f(x_0+re_*)-\f(x_0)$. Then $g$ is $1$-Lipschitz and $g(\pm R)=\f(x_R^\pm)-\f(x_0)=\pm R$ (by Claim 3). So $g(r)=r$ for every $r\in (-R,R)$ yielding $1=g'(r)=e_*\cdot \nabla \f(x_0+re_*)$ for every $r$. Thus, $\nabla \f(x_0+re_*)=e_*$ for every $r\in [-R,R]$, in particular, $e_*=\nabla \f(x_0)$, i.e.,
 $X(r)=x_0+r\nabla \f(x_0)$ is a characteristic of \eqref{cara} and 
 $$\nabla \varphi(X(r))=\nabla \varphi\left(x_{0}\right), \f(x_0+r\nabla \f(x_0))=\f(x_0)+r, \quad \forall r \in[-R, R].$$ In particular, the (unique) maximum and minimum of $\f$ in 
 $\bar B_R(x_0)$ are achieved at the points $x_R^\pm=x_0\pm R\nabla \varphi(x_{0})$.
 
 It remains to prove that $\nabla \f$ is continuous in $\Omega$. Indeed, let $x_n\to x_0$ in $\Omega$ and $\bar B_R(x_n)\subset \Omega$ for large $n$. Up to a subsequence, we may assume that $\nabla \f(x_n)\to e\in \Ss^{N-1}$. By above, we know that $\f(x_n+R\nabla \f(x_n))=\f(x_n)+R$. 
 Passing to the limit, we obtain $\f(x_0+Re)=\f(x_0)+R$, meaning that $x_0+Re$ is the maximum of $\f$ in 
$\bar B_R(x_0)$. By uniqueness of the maximum point $x_R^+$, we conclude that $e=\nabla \f(x_0)$. The uniqueness of the limit $e$ for such subsequences yield the convergence of the whole sequence $(\nabla \f(x_n))_n$ to $\nabla \f(x_0)$.
\end{proof}

\bigskip

\begin{proof}[Proof of Theorem \ref{thm1}]  Let $B$ be a ball, $\bar B\subset \Omega$ and we consider $d \in (0, \frac{\operatorname{dist}(B, \partial \Omega)}5)$. We will prove the following:

\begin{claim} \label{cl2} There exists a universal constant $C>0$ such that $$|\nabla \varphi(x)-\nabla \varphi(y)| \leqslant \frac{C}{d}|x-y|,\quad  \forall x, y \in B \, \textrm{with} \, |x-y|<\frac d{10}.$$
\end{claim}

\begin{proof} Let $X_x$ and $X_y$ be the characteristics passing through $x$ and $y$ constructed in Lemma \ref{lem1} (that are lines in direction $\nabla \f(x)$ and $\nabla \f(y)$). If $X_x$ and $X_y$ coincide inside $B$ (in particular, $\nabla \f(x)=\pm  \nabla \f(y)$ by Lemma \ref{lem1}), then Lemma \ref{lem1} implies that $\nabla\f(x)=\nabla \f(y)$ and the claim is trivial in that case. Otherwise, $\nabla \f(x)\neq \pm  \nabla \f(y)$ and $X_x$ and $X_y$ cannot intersect inside $\Omega$ (as $\nabla \f$ is continuous in $\Omega$ by Lemma \ref{lem1}). Let $|x-y|<\frac d{10}$ and $x'$ be the projection of $y$ on $X_x$. Clearly, 
$$\ell=|x'-y|\leq |x-y|<\frac{d}{10},$$ $\operatorname{dist}(x', \partial \Omega)>5d-\ell$ and $\nabla \f(x')=\nabla \f(x)$. Up to changing $\f$ in $-\f$, we may assume that $\f(x')\leq \f(y)$ and up to an additive constant, we also may assume that $0=\f(x')\leq \f(y)=a$. As $|\nabla \f|=1$ on the segment $[x'y]$, we deduce that $0\leq a=\f(y)-\f(x')\leq |x'-y|=\ell$. Let $z$ be the point on the characteristic $X_y$ reached at time $t=d$ in direction  $\nabla \f(y)$, i.e., $|z-y|=d$, $\nabla \f(y)=\frac{z-y}d$ and $\f(z)=a+d$; in particular, $z\in \Omega$. Let $w$ be the point on the characteristic $X_{x'}$ reached at time $t=-d$ in direction  $\nabla \f(x')$, i.e., $|x'-w|=d$, $\nabla \f(x')=\frac{x'-w}d$ and $\f(w)=-d$; in particular, $w \in \Omega$. We deduce that
$$|z-w|=\int_{[zw]} |\nabla \f|\, d\h^1\geq \left| \int_{[zw]} \nabla \f \cdot \frac{z-w}{|z-w|}\, d\h^1\right|=|\f(z)-\f(w)|=2d+a\geq 2d.$$ 
By Pitagora's formula in the triangle $x'yw$ we have 
$|y-w|^2=\ell^2+d^2$.  
Denoting $\alpha$ the angle between $\vec{wy}$ and $\nabla \f(y)$, the cosine formula in the triangle $wyz$ yields 
\begin{align*}
-\cos \alpha&=\frac{|y-w|^2+|y-z|^2-|z-w|^2}{2|y-z|\cdot |y-w|}\\
&=
\frac{2d^2+\ell^2-|z-w|^2}{2d\sqrt{d^2+\ell^2}}\leq -\frac{2d^2-\ell^2}{2d\sqrt{d^2+\ell^2}}=-1+O(\frac{\ell^2}{d^2}).
\end{align*} 
As $\alpha\in (0, \frac\pi2)$, it yields $\sin^2 \alpha=1-\cos^2 \alpha\leq 2(1-\cos \alpha)=O(\frac{\ell^2}{d^2})$. So, $\sin \alpha=O(\frac{\ell}d)$. Denoting $\beta\in (0, \frac\pi2)$ the angle between $\vec{wy}$ and $\nabla \f(x)$, we compute in the triangle $x'yw$:
$$\sin \beta\leq \tan \beta=\frac{|y-x'|}{|x'-w|}=\frac{\ell}d.$$ In particular, $0\leq \alpha+ \beta\leq \frac\pi2$ if $\ell<\frac{d}{10}$.  Finally, denoting $\gamma$ the angle between $\nabla \f(x)$ and $\nabla \f(y)$, the triangle inequality yields $$\sin \gamma\leq \sin (\alpha+\beta)\leq \sin \alpha+\sin \beta=O(\frac{\ell}d).$$ 
We conclude
$$|\nabla \varphi(x)-\nabla \varphi(y)|=2\sin \frac\gamma2=O(\sin \gamma)=O(\frac{\ell}d)\leq \frac{C}{d}|x-y|$$
for some universal $C>0$.
\end{proof}

The conclusion of Theorem \ref{thm1} follows.
\end{proof}

\medskip

\noindent {\bf Acknowledgement}.
The author thanks Xavier Lamy for very useful discussions. The author  is partially supported by the ANR projects ANR-21-CE40-0004 and
ANR-22-CE40-0006-01.

\end{document}